\documentclass[12pt,twoside,a4paper,english,reqno]{amsart}
\usepackage{listings,graphicx,amsmath,varioref,amscd,amssymb,color,bm,stmaryrd}
\usepackage[pdftex,colorlinks=false,backref]{hyperref}
\usepackage{amsthm}

\newlength{\defbaselineskip}
\theoremstyle{plain}
\newtheorem{theorem}{Theorem}[section]

\newtheorem{lemma}[theorem]{Lemma}

\theoremstyle{definition}
\newtheorem{definition}[theorem]{Definition}

\newtheorem{remark}[theorem]{Remark}

\theoremstyle{remark}

\newcommand{\re}{\mathbb{R}}

\def\rn{\mathbb{R}^{N}}

\def\be{\begin{equation}}
\def\ee{\end{equation}}
\def\rife#1{(\ref{#1})}

\def\t1p0{T^{1,p}_{0}(\Omega)}

\def\m2{M^{\frac{N(p-1)}{N-1}}(\Omega)}

\def\rn{\mathbb{R}^{N}}

\def\into{\int_{\Omega}}

\def\w-1p'{W^{-1,p'}(\Omega)}
\def\pw-1p'u{L^{p'}(0,1;W^{-1,p'}(\Omega))}

\def\dys{\displaystyle}

\def\lp'n{(L^{p'}(\Omega))^{N}}

\def\bk{\color{black}}

\long\def\salta#1{\relax}

\begin{document}


\author{K. H. Karlsen}
\address{(Kenneth H. Karlsen) \\
Centre of Mathematics for Applications,
University of Oslo, P.O. Box 1053, Blindern
NO-0316,  Oslo, Norway}
\email{kennethk@math.uio.no}

\author{F. Petitta}
\address{(Francesco Petitta) \\
Dpto. de An\'alisis Matem\'atico,
Universitat de Val\'encia,
C/ Dr. Moliner 50, 46100,
Burjassot, Valencia, Spain}
\email{francesco.petitta@uv.es}

\author{S. Ulusoy}
\address{(Suleyman Ulusoy) \\
Center for Scientific Computation and Mathematical Modeling
(CSCAMM), University of Maryland, College Park, MD, 20742, USA}
\email{sulusoy@cscamm.umd.edu}


\begin{abstract}
We describe a duality method to prove both existence and  uniqueness \bk of solutions
to nonlocal problems like
$$
(-\Delta)^s v = \mu \quad \text{in}\ \mathbb{R}^N,
$$
with vanishing conditions at infinity. 
 Here $\mu$ is a bounded
Radon measure whose support is compactly
contained in $\mathbb{R}^N$, $N\geq2$, \bk
and $-(\Delta)^s$  is the fractional Laplace operator
of order $s\in (1/2,1)$.
\end{abstract}

\title[The fractional Laplacian with measure data]{A duality approach
to the fractional Laplacian with measure data}

\date{\today}

\maketitle


\section{Introduction}

The study of elliptic equations with measure data (or with $L^1-$data)
is motivated by some engineering problems. See for
instance \cite{8, 9, 10, 13} for applications in
electromagnetic induction heating, modeling of wells in porous
media flow, and  and the $k-\varepsilon$ model of turbulence.

For $N>2,$ \bk where $N$  denotes the dimension, the main mathematical
difficulty in such kind of problems is that, it is not always possible to
employ the classical variational methods for these problems.

Integro-partial differential equations, also referred to as
fractional or L\'evy partial differential equations, appear
frequently in many different areas of research and find many
applications in engineering and finance, including nonlinear
acoustics, statistical mechanics, biology, fluid flow, pricing of
financial instruments, and portfolio optimization. Though  different
 approaches to these kind of problems are possible (e.g., Harmonic Analysis tools), 
in this note we propose a purely PDE method which is easy and flexible
enough  to be utilized in a variety of different contexts
involving linear nonlocal operators and irregular 
data.

Motivated by the applications mentioned above, for the 
sake of simplicity, here we focus on  
the study of the fractional Laplace problem
\begin{equation}\label{main}
	\begin{array}{l}
		(-\Delta)^s u = \mu \quad \text{in $\mathbb{R}^N$},  \\\\
		u(x)\to 0 \quad \text{as $|x|\to +\infty$}.
	\end{array}
\end{equation}
where $\mu$ is a bounded, compactly supported Radon measure whose
support is compactly contained in $\mathbb{R}^N$, $N\geq 2$.
 
In (\ref{main}) $(-\Delta)^s$ is the fractional Laplacian of order
$s\in (\frac12,1)$ defined, up to constants, as
\begin{equation}\label{mainf}
    (-\Delta)^s v =
    P.V. \int_{\mathbb{R}^N} \frac{v(x)-v(y)}{|x-y|^{N+2s}}\ dy.
\end{equation}

Notice that, due to the singularity of the kernel, the right hand side of the previous
expression is not well defined in general. Due to this fact we have to  restrict ourselves
to functions $v$ belonging to some fractional Sobolev space 
$W^{\eta,p}_{loc} (\rn)$ (see Definition \rife{fsp} below). In this context, as we will see,
it is possible to give a precise meaning to the previous expression.

Due to the nonlocal character of the operator, we set the problem in the
whole $\rn$, so that some of the arguments turn out to have 
their own, not surprising, interpretation, also in this limit  case $s=1$.

Our main result is the following:

\begin{theorem}\label{maint}
Let $\mu$ be a Radon measure with compact support in $\rn$. Then
there exists a  unique \bk duality  solution $u$ (see definition
\ref{duale} below) for problem \rife{main}.  Moreover, $u\in W^{1-
\frac{2-2s}{q},q}_{loc}(\rn)$, for any $q<\frac{N+2-2s}{N+1-2s}$.
\end{theorem}

We emphasize that the value of $q$ is always less than $2$, and that
this is consistent with the classical theory of  Dirichlet problems \bk concerning linear elliptic
equations with measure data on bounded domains (see for instance \cite{s}).
Furthermore, the choice to restrict ourselves to the case $1/2<s<1$
is made simply to clarify the statements and the proofs. Notice  that, with our method, 
it is possible to treat the case of values $s$ that are
somewhat smaller than $\frac12$. Indeed, we can allow for $s>\frac14
(2+N - \sqrt{4+N^2})$,  that in particular  implies $1-
\frac{2-2s}{q}>0$.

On the other hand, the case $s=1$ fall into the classical framework 
in the case of Dirichlet boundary value problems on bounded domains (see \cite{s} again).   
Concerning the data, we will only consider bounded measures 
which are compactly supported in $\rn$. This is crucial since 
our methods will rely on the use of auxiliary functions solving 
associated \emph{dual problems} (namely the Riesz Potentials 
for the Fractional Laplace operator), 
so that the key role will be played by local estimates in suitable fractional 
Sobolev spaces gathered together with vanishing 
condition at infinity for these functions.

The  paper is organized as follows: In Section \ref{sec2}
we introduce some preliminary facts and prove a first existence
result.  In Section \ref{sec:first-existence} we establish a result regarding 
the existence of a duality solution in a suitable Lebesgue space. 
We conclude the proof of  our main result, 
Theorem \ref{maint}, in Section \ref{sec3}.

\section{Preliminary facts and dual problem}\label{sec2}

We start this section by recalling some basic facts 
about \emph{fractional Sobolev spaces}. 
For further details see \cite{Dib}, \cite{st} (see also \cite{lt}).

\begin{definition}\label{fsp}
For $0<\eta< 1$ and $1\leq p<\infty$,  let $\Omega$ be an 
open domain of $\rn$. We define the \emph{fractional sobolev space} 
$W^{\eta,p}(\Omega)$ as the set of all functions $u$ in $L^p (\Omega)$ 
such that
 $$
 \into\into \frac{|u(x)-u(y)|^p}{|x-y|^{N+\eta p}}\ dxdy<\infty.
 $$
This space, endowed with the norm
$$
\|u\|_{W^{\eta,p}(\Omega)}=\|u\|_{L^p (\Omega)} 
+ \left( \into\into \frac{|u(x)-u(y)|^p}{|x-y|^{N+\eta p}}\ dxdy\right)^{\frac{1}{p}}
$$
is a Banach space. Moreover, we will say that $u$ is in
$W_{loc}^{\eta,p}(\rn)$ if $u\in W^{\eta,p}(B_R)$, for any ball
$B_R$ in $\rn$.
\end{definition}

Let us observe that, keeping an eye on the definition of fractional derivatives using
the Fourier transform, we  formally have  
$$
|D|^{\eta}u \in L^{p}(\Omega) \iff \iint_{\Omega \times \Omega}
\frac{|u(x)-u(y)|^p}{|x-y|^{N+\eta p}} \ dxdy < \infty.
$$

In our analysis, we shall also use the following Sobolev embedding theorem
for fractional order spaces which we recall here in the form we need
it, see for instance \cite {A, st}.

\begin{theorem}[Fractional Sobolev Embedding Theorem]\label{emb}
Let $B_{R}$ be a ball of radius $R$  in $\rn$.
Then, there exists a constant $C$ depending only on $\eta$ and $N$, such that
\begin{equation}\label{gamma}
	\|v\|_{L^{\overline{\gamma}}(B_R)}\leq C\|v\|_{W^{\eta, p}(B_{R})}, 
	\qquad  v\in C^{\infty}_{0}(B_{R}),
\end{equation}
 where $\overline{\gamma}=\frac{Np}{N-\eta p}$, with $p>1$ and $0<\eta<1$.
Moreover, $W^{\eta, p}(B_{R})$ is compactly embedded in
$L^\gamma (B_R)$ for any $1\leq \gamma<\overline{\gamma}$.
\end{theorem}

Let us come back to our nonlocal problem. We fix $\frac12< s< 1$. 
We focus on the following elliptic nonlocal problem
\begin{equation}\label{du}
\begin{split}
& (-\Delta)^{s}u =\mu  \ \ \ \text{in }\  \rn   \, \\
& u(x) \to  0, \quad \text{as}\  |x| \to \infty.
\end{split}
\end{equation}

If $\mu$ is a smooth function, then a vast amount of theory has been
developed for this type of problems \cite{Garroni:2002il,st,lan,sil} and 
we refer to those for further details.  Let us just recall that, if $\mu=f$ is 
sufficiently regular, for instance if  $f$ belongs to the Schwartz class $\mathcal{S}$, 
then a representation formula is available for solutions of
\rife{du}, through the convolution with the Poisson kernel, namely
there exists a constant $C_{N,s}$, such that
\begin{equation}\label{rep}
	w(x)= (-\Delta)^{-s}f= C_{N,s}\int_{\rn}\frac{f(y)}{|x-y|^{N-2s}}\ dy,
\end{equation}
or, equivalently 
\begin{equation}\label{doublew}
	(-\Delta)^{s}w = f \ \ \ \text{in }\  \rn\, . 
\end{equation}
The function $w$ is also called the Riesz potential 
of order $s$ associated to the function $f$. 

Moreover, it is easy to check that, if for instance $f$ is 
compactly supported in $\rn$, then $w(x)\to 0$ as $|x|\to \infty$.

For our aims, we need to specify a bit this general fact.
In particular, we need a preliminary result concerning
the regularity of  functions $w$ defined through \rife{rep}
with less smooth $f$'s, since these functions will be
involved as auxiliary functions in our methods. 

In the next lemma we collect some further properties of 
the Riesz potentials $w$ defined in \rife{rep} in the case of   $f \in L^\sigma(\rn)$, with
$\sigma>\frac{N}{2s}$, and compactly supported in $\rn$.

\begin{lemma}\label{w}
Let $w$ be defined as in \rife{rep}, and $f$ in $L^{\sigma}(\rn)$,
with $\sigma>\frac{N}{2s}$. Moreover, assume that there exists
$R>0$ such that $f\equiv 0 $ a.e.~in $B_{R}^{c}$, the complement of
$B_{R}$ in $\rn$. Then $w$ satisfy
\begin{itemize}
\item[i)] $\|w\|_{L^{\infty}( \rn)}\leq C_{N,s,R, \sigma}\|f\|_{L^{\sigma} (\rn)}$;

\item[ii)] $w$ is continuous on $\rn$;
\item[iii)] $\lim\limits_{|x|\to+\infty}w(x)=0$.
\end{itemize}
\end{lemma}

\begin{proof}
From the definition of $w$, using H\"older's inequality, we have
\be\label{pic}
|w(x)|\leq \|f\|_{L^{\sigma}(\rn)}\left(\int_{B_R}
\frac{dy}{|x-y|^{(N-2s)\sigma'}}\right)^{\frac{1}{\sigma'}}\leq C\|f\|_{L^{\sigma}(\rn)}
\ee
since $\sigma>\frac{N}{2s}$ implies $(N-2s)\sigma'<N$, and so i) is proved.

To prove ii) just let $x_n\to x$ in $\rn$. Then, it is easy to check that, since
the Poisson Kernel is radially symmetric, the sequence $\frac{f(y)}{|x_n -y|^{N-2s}}$ is
equiintegrable on $B_R$ and so we can pass to the limit in
$n$ proving the continuity of $w$.

Now, observe that for a given $\varepsilon >0,$ we 
can choose $|x|$  large enough such that
$$
\max_{y \in B_R}   |x-y|^{-(N-2s)\sigma'} \leq \varepsilon.
$$
This clearly implies, from \rife{pic}
$$
|w(x)| \leq C \varepsilon,
$$
for $|x| >>R$ that is  $iii)$.
\end{proof}

Now we are in position to come back  to the 
original problem \rife{du}. As we said,  in order to deal
with rough data, we will develop a duality argument as the one
introduced by Stampacchia in \cite{s} in order to treat linear
elliptic operators in divergence form with Dirichlet boundary conditions.
Here is the definition adapted to our case:

\begin{definition}\label{duale}
We say that  a function $ u\in L_{loc}^{1}( \rn)$  is
a \emph{duality solution} for problem \rife{du} if
\begin{equation}\label{dualf}
\int_{\rn} u g \ dx= \int_{\rn} w\ d\mu,
\end{equation}
for any $g\in C^{\infty}_0(\rn)$, where $w$ is defined by
\begin{equation}\label{dual}
w(x)= C_{N,s}\int_{\rn}\frac{g(y)}{|x-y|^{N-2s}}\ dy,
\end{equation}
 \end{definition}

\begin{remark}\label{coin2}  It is important to observe that $w$ 
defined through \rife{dual} is a duality solution of problem
$$
 (-\Delta)^{s}w =g  
 \ \ \ 
 \text{in }\  \rn\,, 
 \ \ \ \
 w(x) \to  0,   \text{as}\  |x| \to \infty\,.
$$
Indeed, if $h\in C^{\infty}_0(\rn)$, then one can check that
$$
\begin{array}{l}
\displaystyle\int_{\rn} w(x) h(x)  \ dx= 
\int_{\rn}\int_{\rn}C_{N,s}\frac{g(y)h(x)}{|x-y|^{N-2s}}\ dy\ dx
\\\\ \displaystyle= \int_{\rn} g(x)\left(C_{N,s}\int_{\rn}\frac{h(y)}{|x-y|^{N-2s}}\ dy\right)\ dx\,,
\end{array}
$$
that is the definition of \emph{duality solution} for $w$.   This fact allows us to better 
specify the meaning of "$u\to0$ as $|x|\to\infty$" for a duality solution with measure data. 
As in the classical case, even if not explicitly stated in the definition, the 
decay at infinity for the solution is, in some sense, hidden in the 
formulation through the presence of the dual functions $w$ and 
it turns out to be attained in the classical sense 
once the data are regular enough (see iii) of Lemma \ref{w}).
\end{remark}

\begin{remark}\label{coin}
Let us also notice that, if everything is smooth, then a duality solution turns
out to fall into the classical framework. For instance, it coincides 
with the distributional solution (see for instance \cite{ka} for 
its definition for the same problem. In fact, if $\mu=f$ is smooth, 
then for every $\varphi\in C^{\infty}_{0}(\rn)$ we can formally compute
$$
\begin{array}{l}
	\dys \int_{\rn}f\varphi\ dx = \int_{\rn}u g 
	=\frac12 \int_{\rn}\int_{\rn} 
	\frac{(u(y)-u(x) )(\varphi(y)-\varphi(x))}{|x-y|^{N+2s}},
\end{array}
$$
where $(-\Delta)^{s}\varphi =g$. 
That is, according with the  definition
$$
(-\Delta)^s u = f,
$$
in the distributional sense.
\end{remark}

\section{A First Existence Result}\label{sec:first-existence}
Let us first  prove the existence of
a duality solution in a suitable Lebesgue space.

\begin{theorem}\label{te}
Let $\mu$ be a compactly supported Radon measure on $\rn.$ Then 
there exists a  unique \bk duality solution $u$ for problem \rife{du}.
Moreover, $u\in L^{r}_{loc}(\rn)$, for any $r<\frac{N}{N-2s}$.
\end{theorem}

\begin{remark}
Notice that if $s\to 1^-,$ we recover the classical optimal
summability of elliptic equations with measure data,  since in this
case we have  $u\in L^{q}_{loc}(\rn)$, for any $q< \frac{N}{N-2}$.
Moreover,   the fundamental solution for the fractional Laplacian
$\frac{c}{|x|^{N-2s}}$ (see for instance \cite{cs}) belongs to the
same space   around the origin. Therefore, since this is a solution
with $\mu=\delta_0$ we have that our result is optimal.

Also observe that, in view of Theorem \ref{te} (and also Lemma \ref{w}),
Definition \ref{duale} makes sense (by density) for
test functions $g$ not only in $C^\infty_0 (\rn)$ but also 
in $L^{\sigma}_{loc}(\rn)$ with $\sigma>\frac{N}{2s}$.
\end{remark}

\begin{proof}[Proof of Theorem \ref{te}]
Let $R>0$ and consider a ball $B_{R}$ of radius $R$ such
that  supp$(\mu)\subset B_{R}$. For any $g\in C^{\infty}_{0}(B_{R})$, let us
define the following operator $T:C^{\infty}_{0}(B_{R})\mapsto \re$ through
$$
T(g):=\int_{\rn} w(x)\ d\mu.
$$
Thanks to Lemma \ref{w}, $T$ is well defined, and we can write
$$
|T(g)| \leq \|w(x)\|_{L^{\infty}(\rn)}\ |\mu|(\rn)\leq C\|g\|_{L^{\sigma}(B_{R})},
$$
where $C$ depends only on $\mu, N, R, s$ and $\sigma$. 
Then for a fixed $\sigma$, $T$ extends to a bounded continuous
linear functional on  $L^{\sigma}(B_{R})$, so that by 
Riesz Representation Theorem, there 
exists a unique function \bk $u\in L^{\sigma'}(B_{R}) $ such that
\begin{equation}\label{sig}
	\int_{\rn} w\ d\mu=\int_{\rn}ug.
\end{equation}
Notice that both of the integrals in (\ref{sig}) are 
actually computed on $B_{R}$. By repeating the same argument for any
$\sigma>\frac{N}{2s}$ we find a unique $u\in L^{\sigma'}(B_{R})$ for
all $\sigma'<\frac{N}{N-2s}$ such that \rife{sig} holds.

Now, if $R'>R$, we can follow the same argument to find a
function $\hat{u}  \in L^{\sigma'}(B_{R'})$ for all
$\sigma'<\frac{N}{N-2s}$. An easy application of the fundamental
theorem in the calculus of variations shows that
actually $u=\hat{u}$ a.e. on $B_{R}$ and this concludes the proof.
\end{proof}

\section{Additional regularity and Proof of Theorem \ref{maint}}\label{sec3}

We shall prove some further regularity results for
the duality solution of problem \rife{du}. 
Theorem \ref{maint} will follow immediately 
by combining Theorem \ref{te} with the 
regularity result given in Theorem \ref{fur} below. 
In the proof of Theorem \ref{fur} we employ Young's 
inequality for convolutions:  for the sake of the 
completeness we recall it here in the form we need it. 
Here 
$$ 
(f\ast g)(x) := \int_{\rn} f(x-y) g(y) \, dy.
$$
\begin{lemma}\label{yng}(\cite{LL, MS})
If $f \in L^p(\rn)$, $g \in L^q(\rn)$, $\frac{1}{p}+ \frac{1}{q} 
= 1 + \frac{1}{r}$, with $1 \leq p, q, r \leq \infty.$ Then
$$ 
\|f\ast g\|_{L^r(\rn)} 
\leq \|f\|_{L^p(\rn)} ||g||_{L^q(\rn)}.
$$
\end{lemma}

Observe that the previous result can be applied, as a 
particular case, when $p=1$ provided $r=q$. 
Let us state the main result of this section.

\begin{theorem}\label{fur}
Let $\mu$ be a bounded, compactly supported Radon measure on $\rn$. 
Then the duality solution of problem \rife{du}
found in Theorem \ref {te} belongs to $u\in W^{1-\frac{2-2s}{q},q}_{loc}(\rn)$
for any $q<\frac{N+2-2s}{N+1-2s}$.
\end{theorem}

\begin{proof}
Let us fix $R>0$, such that supp$(\mu)\subset B_R$ and
let us approximate $\mu$ by smooth functions
$f_n\in C_{0}^{\infty}(B_R)$  in  the narrow topology of measures. 
This can be easily obtained by standard convolution arguments. 
Moreover we can choose $f_n$ such that
$\|f_n\|_{L^1 (B_R)}\leq C|\mu|(\rn)$.

Now,  consider
$$
u_n (x) =\int_{B_R} \frac{f_n(y)}{|x-y|^{N-2s}}\ dy\,
$$
where we recall that $u_n$ is a (duality) solution of
problem \rife{du} with $f_n$ as datum (see Remark \ref{coin2}).

Now, we claim that the following inequality is valid.
\begin{equation}\label{bound}
	\|u_n\|_{W_{}^{\eta, q}(B_R)}\leq C,
\end{equation}
where
\begin{equation}\label{eta-q}
	\eta=1-\frac{2-2s}{q}, \ \ \ \ q
	=\frac{\sigma'(N+2 -2s)}{N+\sigma'},
\end{equation}
and $\sigma$ is a real number such
that $\sigma > \frac{N}{2s}$ with
$\frac{1}{\sigma}+\frac{1}{\sigma'} = 1$.

Indeed, let $x \in B_R$ be fixed and $\psi$ be a cut-off function
that is compactly supported in $\rn$ and such that $\psi \equiv 1$
on $B_{2R}.$ Note that, $\psi(x) = \psi (x-y)$ for any $y \in B_R.$

Thanks to this choice, we have, using the classical 
properties of Fourier  transform,  that
\begin{equation}\label{dunglike}
	\left||D|^{\eta} u_n(x) \psi^{}(x)\right| \thicksim
	\left|f_n(x)\ast \frac{\psi^{}(x)}{|x|^{N-2s+\eta}}\right|.
\end{equation}
So that, to obtain an upper bound on the $L_{\text{loc}}^q(\rn)$ norm of 
$|D|^{\eta}u_n,$  we multiply it by the cut-off function $\psi$ as indicated
above. The cut-off argument enables us to use the 
standard Young's inequality for convolutions, see Lemma \ref{yng},
with  $r= \frac{\sigma' (N+2-2s)}{N+\sigma'}, p=1$ and 
thus $q = r = \frac{\sigma' (N+2-2s)}{N+\sigma'}$. 
Notice that,  in order to apply the Lemma \ref{yng}, we 
only need to know that $\frac{\psi}{|x|^{N-2s+\eta}} \in L^{q}(\rn)$, which 
is true since $(N-2s+\eta)q <N$.  Therefore, recalling that $f_n$ is 
bounded in $L^1(B_R)$, we finally obtain
$$
\||D|^{\eta} u_n(x)\|_{L^q(B_R)}\leq C\,.
$$

The bound  on $u_n$ in $L^q (B_R)$ follows in a similar way.  Indeed,
we multiply $u_n$ by the same 
the cut-off function $\psi$ as before,  and we observe 
that $\frac{\psi}{|x|^{N-2s}} \in L^q (\rn)$ where $q$ 
is as in \rife{eta-q} (observe that, from the previous 
calculations, $(N-2s)q<(N-2s +\eta)q<N$). 
Hence, recalling that $f_n$ is uniformly bounded in 
$L^1 (B_R)$ and using Young's inequality 
for convolutions (now with $p=1$ and $r= q$), we get
$$
\|u_n\|_{L^q (B_R)}\leq C\|f_n\|_{L^1(B_R)}.
$$
This concludes the proof of the claim \eqref{bound}.

Since the constant on the right hand side of \rife{bound} is independent
of $n,$ we can find a function $v$ such that $u_n$
converges along a subsequence to $v$
weakly in  $ W^{1-(2-2s)/q,q}(B_R)$ for
any $q< \frac{N+2-2s}{N+1-2s}$. Moreover, thanks
Theorem  \ref{emb} we deduce that a subsequence of $u_n$
converges to $v$ strongly in $L^\gamma_{}(B_R)$ for any 
$1\leq \gamma< 1+\frac{2}{N-2s}$ and almost everywhere. 
By repeating the argument for any $R>0$ we are allowed to 
pass  to the limit  in the duality formulation
for $u_n$ to prove that the limit $v$ solves the problem. 
Finally, by subtracting the formulation of  $v$ from the 
one of $u$ proved in Theorem \ref{te} we 
easily deduce that $u=v$.
\end{proof}

\begin{remark} 
Let us emphasize that $q>1$ is approaching $\frac{N+2-2s}{N+1-2s}$
from below as $\sigma$ approaches $\frac{N}{2s}$ from above.
Notice also  that if $s\to 1^-$ we recover the classical optimal
summability for elliptic boundary value problems with 
measure data (see again \cite{s}), since in this
case we have  $u\in W^{1,q}_{loc}(\rn)$, for any $q< \frac{N}{N-1}$.
Finally, as before, the result is optimal since a 
direct computation shows  that the fundamental
solution for the fractional Laplacian,  $\frac{c}{|x|^{N-2s}}$
belongs to the space $W^{1-\frac{2-2s}{q},q}$, for
any $q< \frac{N+2-2s}{N+1-2s}$  around the origin.
\end{remark}

\medskip

In this note, as a first step to study nonlocal
elliptic equations with measure data, we considered the fractional
Laplace equation \rife{main}. Employing the duality method, as the
operator is linear, first introduced by Stampacchia \cite{s}, we
introduce a solution concept and we prove existence, uniqueness and
regularity of these solutions. The approach utilized in this paper can be 
used for other types of nonlocal operators. Indeed, the study of similar problems 
involving more general nonlocal operators will be the subject 
of future work of the same authors.

\end{document}